\documentclass[a4paper]{amsart}

\usepackage{microtype}
\usepackage{bm}

\usepackage{amssymb}

\usepackage[hidelinks]{hyperref}
\usepackage{amsrefs}

\newcommand{\NN}{\mathbb{N}}
\newcommand{\ZZ}{\mathbb{Z}}
\newcommand{\QQ}{\mathbb{Q}}

\newcommand{\KK}{\mathbb{K}}

\newcommand\br[1]{\left(#1\right)}
\newcommand\abs[1]{\left|#1\right|}
\newcommand\norm[1]{\left\|#1\right\|}
\newcommand\set[1]{\left\{#1\right\}}

\newcommand\landauO[1]{\mathop{O}\left(#1\right)}
\newcommand\landausmallo[1]{\mathop{o}\left(#1\right)}
\newcommand\landauOmega[1]{\mathop{\Omega}\left(#1\right)}
\newcommand\landauTheta[1]{\mathop{\Theta}\left(#1\right)}

\newcommand\myvec[1]{ \mathbf{#1} }

\newtheorem{thm}{Theorem}
\newtheorem{lem}[thm]{Lemma}
\newtheorem{cor}[thm]{Corollary}

\begin{document}

\title[Split Thue equations factorised by linear recurrence sequences]{On split families of Thue equations with linear recurrence sequences as factors}

\author[T. Hilgart]{Tobias Hilgart}
\email{tobias.hilgart@plus.ac.at}



\begin{abstract}
    We consider a parametrised family of Thue equations,
    \[
        (x-G_1(n)\, y) \cdots (x-G_d(n)\, y) - y^d = \pm 1,
    \]
    which was first considered by Thomas to have an explicit set of solutions for parameters $n$ larger than some effectively computable constant.

    In the case where the parameter functions are polynomials belonging to an explicitly described family, this is known to be true. We consider other parameter functions, namely linear recurrence sequences, for which it is not obvious that a similar result holds, and confirm that it does for an explicitly described family of linear recurrence sequences.
\end{abstract}

\maketitle

%
%
\section{Introduction}
\label{sec:intro}

    Thue equations, i.e. integer equations of the form $f(x,y) = m$ for an irreducible homogeneous and at least cubic polynomial $f$, are an interesting object of study. Thue~\cite{thue09} used his improvement of Liouville's original result on the approximatability of algebraic numbers to prove that such equations can have at most finitely many integer solutions. Suddenly, a large class of Diophantine equations was proved to be decidable (compared to Hilbert's 10th problem, the answer to which famously concludes that the class of all Diophantine equations is not). To the best of the author's knowledge, it is currently unknown whether even the class of all bivariate Diophantine equations is decidable.

    But for Thue equations we can do even better. Baker~\cite{baker68} used his celebrated work on lower bounds for linear forms in logarithms to prove Thue's original result in an effective way. There was now an algorithm that gave all solutions to any given Thue equation (by computing an upper bound on the absolute value of any solution); let us call this property effectively solvable.

    Baker's work has inspired many authors to consider various different generalisations of effectively solvable equations, such as dropping the homogeneity condition for certain families~\cite{ellison72} or considering inequalities~\cite{petho87}. One of the first to successfully approach parametrised Thue equations (with positive discriminant), i.e. where the coefficients of the Thue equation are themselves polynomials in one or more variables, was Thomas~\cite{tho90}. Since then, various authors and Thomas himself considered many different parametrised Thue equations (for a survey, see \cite{heub05}).

    Thomas himself considered a parametrised Thue equation of arbitrary degree, where for monic polynomials $p_2(t), \ldots, p_d(t) \in \ZZ[t]$ he looked at
    \begin{equation}\label{eq:splitFamily}
        x(x-p_2(n)\, y) \cdots (x-p_d(n)\, y) + u\, y^d = 1, \qquad u=\pm 1,
    \end{equation}
    which he called a \textit{split family} of Thue equations with factors $p_2, \ldots, p_d$. This family always has the trivial solutions
    \[
        \epsilon \set{ (1, 0), (0, u), (p_2(n) u, u), \ldots, (p_d(n) u, u) },
    \]
    where $\epsilon = 1$ for odd $d$ and $\epsilon \pm 1$ for even $d$. He conjectured that if
    \[
        0 < \deg p_2 < \cdots < \deg p_d,
    \]
    then the split family of Thue equations with factors $p_2, \ldots, p_d$ is effectively solvable and has only the solutions given above. The condition is necessary because there are known examples of split families that do not satisfy this condition and have additional solutions not covered by Thomas' set. In the case of cubic split families, Thomas himself proved the conjecture under some additional technical conditions, which restrict the polynomials to a family he called \textit{regular}.

    Halter-Koch et al.~\cite{hakolepeti99} considered a special case of Thomas' conjecture where $p_2, \ldots, p_{d-1}$ are distinct integers and $p_d$ is an integral parameter. They proved that in this case the Thomas conjecture (for a sufficiently large parameter $p_d$) follows from the Lang-Waldschmidt conjecture.

    Heuberger and Tichy considered a multivariate version of Equation~\eqref{eq:splitFamily}, where now $p_i \in \ZZ[t_1, \ldots, t_r]$ and they allowed for a non-zero first polynomial $p_1$. For $\text{LH}(p)$, which they called the homogeneous part of maximal degree in $p$, they gave conditions
    \begin{enumerate}
        \item\label{cond:heubtichy-degree} $\deg p_1 < \cdots < \deg p_{d-2} < \deg p_{d-1} = \deg p_d$, 
        \item\label{cond:heubtichy-LH} $\text{LH}(p_{d-1}) = \text{LH}(p_d)$, while still $p_{d-1} \neq p_d$, and
        \item\label{cond:heubtichy-growth} for any $p \in \set{p_1, \ldots, p_d, p_d-p_{d-1}}$ there exists constants $t_p, c_p$ such that whenever $t_1, \ldots, t_r \geq t_p$ it holds that
        \[
            \abs{ \text{LH}(p(t_1, \ldots, t_r)) } \geq c_p \cdot \br{\min_{k \in \set{1, \ldots, r}}\set{ t_k }}^{\deg p}.
        \]
    \end{enumerate}
    they proved the effective solubility for all parameters $t_1, \ldots, t_r$, which (for an effectively computable $t_0$ and $\tau$ in an explicit range) satisfy the conditions
    \[
        t_0 \leq \min_{k \in \set{1, \ldots, r}}\set{ t_k }, \qquad \max_{k \in \set{1, \ldots, d}}\set{ t_k } \leq \br{ \min_{k \in \set{1, \ldots, r}}\set{ t_k } }^\tau.
    \]

    Heuberger~\cite{heubConj} later improved the result by very technical but explicit conditions on terms involving the degrees of the polynomials. In the cubic case, his conditions are weaker than in Thomas' original result, thus improving his theorem.

    The "polynomial case" is solved to the degree described above. One way to extend these investigations is to consider classes of parameter functions other than polynomials. In~\cite{hilgart22}, an explicit cubic split family was considered, that was parametrised by the Fibonacci and Lucas sequences. The family was solved effectively and the (analogue of the) conjecture observed. Using a combination of different reduction methods, the bounds on the size of the parameters were reduced to such an extent that the remaining equations could be checked, thus solving the equation
    \[
        x(x-F_n\, y)(x-L_n\, y) - y^3 = \pm 1
    \]
    completely. Apart from the cases $n=1$ and $n=3$, only the trivial solutions were found.

    In [2], the ideas extended from the polynomial case and introduced to tackle the exponential case were used to solve the cubic case in general, in an analogue of Thomas' original work. The conditions on the linear recurrence sequences $A(n)$ and $B(n)$ (both with a dominant root) are very mild and only restrictive at all if the dominant roots have the same absolute value. This is possible because the growth of the solutions can be described by $2\times 2$ matrices and their determinants, allowing very few term cancellations. If we consider families of higher degree, there is much more potential for term cancellations, which in turn requires a restriction of the class of parameter functions.

    In the light of Heuberger and Tichy's result \cite{heubtichy99}, our main result is as follows.    
    \begin{thm}\label{thm:main}
        Let $(G_1(n))_{n \in \NN}, \ldots, (G_d(n))_{n \in \NN}$ be $d$ simple linearly recurrent integer sequences satisfying the following conditions.
        \begin{enumerate}
            \item $G_1, \ldots, G_d$ each satisfy a dominant root condition, with dominant roots $\gamma_1, \ldots, \gamma_{d}$ and $0 \leq \gamma_1 < \cdots < \gamma_{d-2} < \gamma_{d-1} = \gamma_d =\gamma$.
            \item The constant terms $g_{d-1}$ and $g_d$ corresponding to the dominant roots $\gamma_{d-1}$ and $\gamma_d$ in the closed formula for $G_{d-1}$ and $G_d$ satisfy $g_{d-1} = g_d = g$.
            \item \label{cond:seconddominantroots} $G_{d-1}$ and $G_d$ have second dominant roots $\delta_{d-1}$ and $\delta_d$ with corresponding constant terms $h_{d-1}$ and $h_d$ satisfying $\abs{\delta_{d-1} } < \abs{\delta_d} < \gamma_{d-2}$ and $\gamma_{d-2}^2 < \gamma \abs{\delta_d}$.
        \end{enumerate}
    
        Define for each $n \in \NN$ the homogeneous polynomial
        \[
            f_n(x, y) = (x-G_1(n)\,y) \cdots (x-G_d(n)\, y) - y^d,
        \]
        and let $x, y, n$ be integers satisfying $\abs{y} \geq 2$ and $f_n(x, y) = \pm 1$. Then there exists an effectively computable constant $\kappa$, depending on the coefficients of $G_1, \ldots, G_d$, such that
        \[
            \max\set{ \log\abs{x}, \log\abs{y}, n } \leq \kappa.
        \]
    \end{thm}

    If we consider the properties of solutions of the equation $f_n(x, y) = \pm 1$ where $\abs{y} \leq 1$, then the above theorem immediately implies
    \begin{cor}
        Let $G_1, \ldots, G_d$ satisfy the Conditions in Theorem~\ref{thm:main}. Then there exists an effectively computable $n_0$ such that for $n \geq n_0$ the parametrised Thue equation
        \[
            (x - G_1(n)\, y) \cdots (x - G_d(n)\, y) - y^d = \pm 1
        \]
        only has the solutions
        \[
            \set{ ( \pm 1, 0), \pm(G_1(n), 1), \ldots, \pm(G_d(n), 1) }.
        \]
        
    \end{cor}

    If we compare the conditions in Theorem 1 or those in the result of Heuberger and Tichy with Thomas' original conjecture, we notice that the condition of sufficiently different growth of the parameter functions (by strictly increasing degrees or dominant roots) is modified. Even the cubic case for linear recurrence sequences suggests that it should be easier in some sense if all the dominant roots are different. However, while Heuberger was able to be less restrictive and closer to the original conjecture in the polynomial case, a key part of his proof is that the variables controlling the growth of the parameter functions, the degrees of the polynomials, are integers. This is in contrast to the situation in Theorem 1, where the controlling terms are the dominant roots (or their logarithms). Forcing these to be integers, so that similar ideas to those in \cite{heubConj} could be applied, would be even more restrictive than the present conditions.

    For simplicity, we assume that all dominant roots are positive real numbers. This is not a restriction, we can apply the theorem to alternating sequences by considering the positive and negative sub-sequences separately.
    
    We use the standard $O$ notation to describe the asymptotic behaviour in terms of $n\to \infty$ and write $f(n) = \landauO{g(n)}$ if, for some positive constants $c$ and $n_0$, we have $\abs{f(n)} \leq c\, g(n) $ for all $n\geq n_0$, $f(n) = \landauOmega{g(n)}$ for the other inequality, and $f(n) = \landauTheta{g(n)}$ if both $f(n) = \landauO{g(n)}$ and $f(n) = \landauOmega{g(n)}$ holds. If $f(n)$ does not influence the asymptotic behaviour of $g(n)$ and we care not to quantify the specifics, we succinctly write $f(n) = \landausmallo{g(n)}$, e.g. we write $x + x^{-1} = x + \landausmallo{x}$ if we care not that the second addend is precisely $x^{-1}$. We use this so that some error terms do not become unnecessarily complicated without adding anything of conceptual or technical relevance.

    For example, using this notation, we immediately obtain the following asymptotic bounds from the requirements in Theorem~\ref{thm:main}.
    \begin{equation}\label{eq:errorterm}
        G_i(n) = \landauOmega{ \gamma_i^n }, \qquad \abs{G_i(n) - G_j(n)} = 
        \begin{cases}
            \landauOmega{\abs{\delta_d}^n} \quad \text{ if }\set{i, j} = \set{d-1, d} \\
            \landauOmega{ \max\set{\gamma_i, \gamma_j}^n } \quad \text{ else}
        \end{cases}
    \end{equation}

    Other more involved statements about the behaviour of  the recurrences and the roots of the polynomial $f_n$ will be discussed in a separate section before proceeding to the proof of Theorem~\ref{thm:main} in the following section.

%
%
\section{Auxiliary Results}
\label{sec:lemmas}

    Let $\alpha^{(1)} = \alpha^{(1)}(n), \dots \alpha^{(d)} = \alpha^{(d)}(n)$ be the roots of the polynomial $f_n$. Since $f_n(x) = f_n(x, 1) = (x-G_1(n))\cdots(x-G_d(n)) - 1$, we expect $\alpha^{(1)}$ to be close to $G_1(n)$, $\alpha^{(2)}$ to be close to $G_2(n)$, and so on. We quantify this in the next lemma, making sure that we also have an explicit term for the expressions $\alpha^{(i)} - G_i(n)$; the subsequent errors are of little interest and only relevant for the proof of the lemma.
    \begin{lem}\label{lem:rootapprox}
        The roots $\alpha^{(1)}, \ldots \alpha^{(d)}$ are all real and we have for $i=1, \ldots, d$ and
        \[
            \gamma_\epsilon = \gamma_\epsilon(i) = 
            \begin{cases}
                \gamma_i^{i-1} \cdot \prod_{k=i+1}^{d-2} \gamma_k \cdot \gamma^2 \quad \text{ if } i \not\in\set{d-1, d} \\
                \gamma^{d-2} \delta_d  \quad \text{ else}
            \end{cases}
        \]
        that
        \[
            \alpha^{(i)} = G_i(n) + \frac{ 1 + \landauO{\gamma_\epsilon^{-n}} }{ \prod_{ \substack{k = 1 \\ k\neq i} }^d (G_i(n) - G_k(n)) }.
        \]
    \end{lem}
    \begin{proof}
        For $u = \pm 1$ we plug our approximation for the root $\alpha^{(i)}$ into $f_n$, replacing the $O$-term by $u\, \gamma_\epsilon^{-n}$. The sign of $u$ determines the sign of the expression, the statement then follows from the intermediate value theorem.

        Let $i \in \set{1, \ldots, d}$, then we consider the expression $f(\xi^{(i)}) + 1$ for
        \[
            \xi^{(i)} = G_i(n) + \frac{ 1 + u\, \gamma_\epsilon^{-n} }{ \prod_{ \substack{k = 1 \\ k\neq i} }^d (G_i(n) - G_k(n)) }.
        \]

        The form of our function is $f(\xi^{(i)})+1 = (\xi^{(i)} - G_1(n)) \cdots (\xi^{(i)} - G_d(n))$ and we split the product into the factor $(\xi^{(i)} - G_i(n))$, which is
        \[
            \frac{ 1 + u\gamma_\epsilon^{-n} }{ \prod_{ \substack{k = 1 \\ k\neq i} }^d (G_i(n) - G_k(n))  },
        \]
        and everything else. Note that by the definition of $\gamma_\epsilon(i)$ and Equation~\eqref{eq:errorterm} we have
        \[
            \prod_{ \substack{k = 1 \\ k\neq i} }^d (G_i(n) - G_k(n)) = \landauOmega{ \gamma_\epsilon^n }.
        \]
        
        The remaining product is
        \[
            \prod_{ \substack{k = 1 \\ k\neq i} }^d (\xi^{(i)} - G_k(n)) = \prod_{ \substack{k = 1 \\ k\neq i} }^d \left( G_i(n) - G_k(n) + \frac{ 1 + u\, \gamma_\epsilon^{-n} }{ \prod_{ \substack{k = 1 \\ k\neq i} }^d (G_i(n) - G_k(n)) } \right),
        \]
        and we see each factor as made up of two addends. The first is $(G_i(n) - G_k(n))$, while the second is a $\landauO{\gamma_\epsilon^{-n}}$. If we multiply the product, the highest order term is the product of all the $d-1$ many $(G_i(n) - G_k(n))$, followed by the sums of those where we have $(G_i(n)-G_k(n))$ $d-2$ times and $\landauO{\gamma_\epsilon^{-n}}$ once, and so on.
        
        We explicitly write the two highest order terms (in the sense described above), the rest we can hide in an error term $\landausmallo{\gamma_\epsilon^{-n}}$, i.e.
        \begin{align*}
             \prod_{ \substack{k = 1 \\ k\neq i} }^d ( G_i(n) - G_k(n) ) &+ \sum_{ \substack{l = 2 \\ l\neq i} }^d \frac{ \prod_{ \substack{k = 1 \\ k\neq i, l} }^d (G_i(n) - G_k(n)) }{ \prod_{ \substack{k = 1 \\ k\neq i} }^d (G_i(n)-G_k(n)) } \left( 1 + u\, \gamma_\epsilon^{-n} \right)             +\landausmallo{\gamma_\epsilon^{-n}}  \\
             = \prod_{ \substack{k = 1 \\ k\neq i} }^d ( G_i(n) - G_k(n) ) &+ \sum_{ \substack{k,l = 1 \\ i \not \in \set{k,l}, k\neq l} }^d \frac{ 1 + u \gamma_\epsilon^{-n} }{ (G_i(n)-G_k(n)) ( G_i(n) - G_l(n) ) } + \landausmallo{\gamma_\epsilon^{-n}}.
        \end{align*}
        
        We can also move the terms with $u \gamma_\epsilon^{-n}$ into the error term $\landausmallo{\gamma_\epsilon^{-n}}$, i.e.
        \begin{equation}\label{eq:rootapprox_prod}
            \prod_{ \substack{k = 1 \\ k\neq i} }^d ( G_i(n) - G_k(n) ) + \sum_{ \substack{k,l = 1 \\ i \not \in \set{k,l}, k\neq l} }^d \frac{1}{ (G_i(n)-G_k(n)) ( G_i(n) - G_l(n) ) } + \landausmallo{\gamma_\epsilon^{-n}}
        \end{equation}

        We multiply with the factor $(\xi^{(i)}-G_i(n))$, the first product is then $1+u \gamma_\epsilon^{-n}$ while the product with the sum is again a $\landausmallo{\gamma_\epsilon^{-n}}$ with the same argument. Put together,
        \[
            f(\xi^{(i)}) + 1 = 1 + u \gamma_\epsilon^{-n} + \landausmallo{ \gamma_\epsilon^{-n} }.
        \]
        applies. Setting $u=1$ and $u=-1$, we get $f(\xi^{(i)}) > 0$ and $f(\xi^{(i)}) < 0$ if $n$ is sufficiently large such that the error term can no longer compensate for $\pm \gamma_\epsilon^{-n}$. From this the statement follows by the intermediate value theorem.
    \end{proof}

    If we combine this with Equation~\eqref{eq:errorterm}, we immediately obtain
    \begin{cor}\label{cor:rootbounds}
        Let $\alpha^{(i)}$ and $\alpha^{(j)}$ be two roots of $f_n$. Then we have
        \[
            \abs{\alpha^{(i)}} = \landauOmega{\gamma_i^n}, \quad \abs{ \alpha^{(i)} - \alpha^{(j)} } = 
            \begin{cases}
                \landauOmega{ \abs{\delta_d}^n } \quad \text{ if } \set{i, j} = \set{d-1, d} \\
                \landauOmega{ \max\set{\gamma_i, \gamma_j}^n } \quad \text{ else},
            \end{cases}
        \]
        and the same result holds true if we replace $\alpha^{(j)}$ with $G_j(n)$.
    \end{cor}

    Next, we examine the number field $\KK_n = \QQ\br{\alpha^{(1)}}$ generated by $f_n$. To do this, we define
    \[
        \eta_j^{(i)} = \alpha^{(i)} - G_j(n) \quad \text{ for } j = 1, \ldots, d,
    \]
    for each $i = 1, \ldots, d$. This definition immediately gives us
    \begin{equation}\label{eq:etamultiplyto1}
        \eta_1^{(i)} \cdots \eta_d^{(i)} = f(\alpha^{(i)}) + 1 = 1
    \end{equation}
    for each $i = 1, \ldots d$. If $i=j$, then by Lemma~\ref{lem:rootapprox} and Equation~\eqref{eq:errorterm} the asymptotic bound  $\abs{\eta_i^{(i)}} = \landauO{ \br{\gamma_2 \cdots \gamma_{d-2} \gamma^2}^{-n} }$ holds for any $i$, since the $\gamma_\epsilon$ defined in the lemma would only swap some of the factors $\gamma_2\cdots \gamma_{d-2}\gamma^2$ for larger ones. If instead $i\neq j$ we have either $\abs{\eta_j^{(i)}} = \landauOmega{\delta_d^n}$ or $\abs{\eta_j^{(i)}} = \landauOmega{ \max\set{\gamma_i, \gamma_j}^n }$ by Corollary~\ref{cor:rootbounds}, which we collect in the following equation,
    \begin{align}\label{eq:etabounds}
        \abs{ \eta_j^{(i)} } =
        \begin{cases}
            \landauO{ \br{\gamma_2 \cdots \gamma_{d-2} \gamma^2}^{-n} } \quad &\text{ if } i=j \\
            \landauOmega{ \delta_d^n } \quad &\text{ if } \set{i, j} = \set{d-1, d} \\
            \landauOmega{ \max\set{\gamma_i, \gamma_j}^n } \quad &\text{ else}.
        \end{cases}
    \end{align}

    We want to make statements about matrices containing logarithms of these $\eta_j^{(i)}$, and to do this we use the following theorem of Gershgorin \cite{gersh31}, sometimes referred to as \textit{Gershgorin's Circle Theorem},
    \begin{thm}\label{thm:circletheorem}
        Let $A = \begin{pmatrix} a_{i j} \end{pmatrix}$ be a $n\times n$ matrix with complex entries and define for each row $i = 1, \ldots, n$ the radius
        \[
            R_i = \sum_{ \substack{j = 1 \\ j\neq i } }^n \abs{ a_{i j} }.
        \]

        Then every eigenvalue $\lambda$ of A lies in at least one of the disks
        \[
            \set{ z \, : \, \abs{z - a_{i i}} \leq R_i } \qquad i=1\ldots, n.
        \]
    \end{thm}

    \begin{lem}\label{lem:etamatrixdet}
        Let $k \in \set{1, \ldots, d-1}$ and let
        \[
            B_k =
            \begin{pmatrix}
                \log\abs{ \eta_1^{(1)} } & \log\abs{ \eta_2^{(1)} } & \cdots & \log\abs{ \eta_k^{(1)} } \\
                \vdots & \vdots & \ddots & \vdots \\
                \log\abs{ \eta_1^{(k)} } & \log\abs{ \eta_2^{(k)} } & \cdots & \log\abs{ \eta_k^{(k)} }
            \end{pmatrix},
        \]
        then we have $\det B_k = \landauTheta{n^k}$.
    \end{lem}
    \begin{proof}
        It follows immediately from Lemma~\ref{lem:rootapprox} that $\log\abs{\eta_j^{(i)}} = \landauTheta{n}$, which implies $\det B_k = \landauO{ n^k }$. We need to prove the other direction, that $\det B_k = \landauOmega{ n^k }$.

        Let $\lambda$ be the smallest eigenvalue of $B_k$. According to Theorem~\ref{thm:circletheorem}, for at least one $i \in \set{1, \ldots, k}$ it must be that
        \[
            \sum_{ \substack{j = 1 \\ j\neq i} }^k \abs{ \log\abs{ \eta_j^{(i)} } } \geq \abs{ \lambda - \log\abs{ \eta_i^{(i)} } }.
        \]

        From this inequality it follows that
        \[
            \abs{\lambda} \geq \abs{ \log\abs{\eta_i^{(i)}} } - \sum_{ \substack{j = 1 \\ j\neq i} }^k \abs{ \log\abs{ \eta_j^{(i)} } },
        \]
        and we can ignore the outer absolute value in the sum, since for each $j\neq i$ we have $\log\abs{\eta_j^{(i)}} > 0$. If we replace $\log\abs{\eta_i^{(i)}}$ by $-\sum_{ \substack{j = 1 \\ j\neq i} }^d \log\abs{ \eta_j^{(i)} }$ using Equation\eqref{eq:etamultiplyto1}, this implies
        \[
            \abs{\lambda} \geq \sum_{ j=k+1 }^d \log\abs{ \eta_j^{(i)} }.
        \]
        We can bound the sum from below by $\log\abs{\eta_d^{(i)}}$, which is a $\landauOmega{n}$ by Corollary~\ref{cor:rootbounds}. And if the smallest eigenvalue $\lambda$ is a $\landauOmega{n}$, then the determinant $\det B_k$ must be a $\landauOmega{n^k}$, proving the lemma.
    \end{proof}

    We next examine the order $\mathcal{O} = \ZZ[\alpha^{(1)}]$ of $\KK_n$ and its Regulator $R_{\mathcal{O}}$. We use the following estimate by Pohst \cite{pohst77}, whose proof, as noted by Heuberger \cite{heub99}, also works verbatim for non-maximal orders.

    \begin{thm}
        Let $\KK$ be a totally real algebraic number field of degree at least $4$, $\mathfrak{D}$ an order of $\KK$ with discriminant $d_{\mathfrak{D}}$. Let $R_{\mathfrak{D}}$ be the regulator of $\mathfrak{D}$. Then there exists an explicit constant $c$, depending only on the degree of $\KK$, such that
        \[
            R_{\mathfrak{D}} \geq c \log\br{ d_{\mathfrak{D}} }.
        \]
    \end{thm}

    For our order $\mathcal{O}$, the previous result and Corollary~\ref{cor:rootbounds} immediately gives
    \begin{cor}\label{cor:regulatorbound}
        We have that
        \[
            R_{\mathcal{O}} = \landauOmega{n}.
        \]
    \end{cor}

    If we form the subgroup of $\mathcal{O}^\times$ generated by $-1$ and $\eta_1^{(i)}, \ldots, \eta_{d-1}^{(i)}$, we get the following estimates in
    \begin{lem}\label{lem:regNindex}
        Consider the subgroup $G = \langle -1, \eta_1^{(i)}, \ldots, \eta_{d-1}^{(i)} \rangle$ of $\mathcal{O}^\times = \ZZ[\alpha^{(1)}]^\times$ with Regulator $R_G$ and Index $I$. Then we have
        \[
            R_G = \landauTheta{n^{d-1}}, \qquad I = \landauO{n^{d-2}}.
        \]
    \end{lem}
    \begin{proof}
        The estimate for the regulator follows from Lemma~\ref{lem:etamatrixdet} for $k=d-1$. The estimate for the index follows from the relation $I = R_G/R_{\mathcal{O}}$ and Corollary~\ref{cor:regulatorbound}.
    \end{proof}

%
%
\section{Proof of Theorem~\ref{thm:main}}
\label{sec:proof}

    We now proceed with the proof of Theorem~\ref{thm:main}. Let $x, y, n$ be integers satisfying $f_n(x, y) = \pm 1$. Note that for $y = 0$ this implies $x^d = \pm 1$ which leads to the solution $(x, y, n) = (\pm 1, 0, n)$ for every $n \in \NN$. If instead $y = \pm 1$, then $f_n(x, y) = \pm 1$ implies either
    \[
        (x-G_1(n)\, y) \cdots (x-G_d(n)\, y) = 0,
    \]
    from which we get the solutions $(\pm G_i(n), 1, n)$ for every $n \in \NN$ and $i = 1, \ldots, d$, or
    \[
        (x-G_1(n)\, y) \cdots (x-G_d(n)\, y) = \pm 2.
    \]
    
    Since the factors on the left are all distinct integers, there are no solutions if we have $d \geq 4$ of them. These are all solutions where $\abs{y} \leq 1$, and from now on we can (and must) assume that $\abs{y} \geq 2$.
    
    We will refer to the terms $x - \alpha^{(i)}y$ by $\beta^{(i)}$ and call $(x,y)$ a solution of type $j$ if
    \[
        \abs{ \beta^{(j)} } = \min\set{ \abs{\beta^{(1)}}, \ldots, \abs{\beta^{(d)}} },
    \]
    which implies by the triangle inequality for $i \neq j$, that
    \begin{equation}\label{eq:betailowerbound}
       2\abs{ \beta^{(i)} } \geq \abs{ \beta^{(i)} - \beta^{(j)} } = \abs{ y \br{ \alpha^{(j)} - \alpha^{(i)} } }.
    \end{equation}
    
    Analogously to Lemma~\ref{lem:rootapprox} and/or in view of Corollary~\ref{cor:rootbounds}, we can define the "correct" error term as $\gamma_\epsilon(j)$, where
    \begin{align*}
        \gamma_\epsilon(j) = 
        \begin{cases}
            \gamma_j^{j-1} \cdot \prod_{i=j+1}^{d-2} \gamma_i \cdot \gamma^2 \quad &\text{ if } j \not\in\set{d-1, d} \\
            \gamma^{d-2} \delta_d \quad &\text{ else}
        \end{cases}
    \end{align*}
    and we note that the factor $\gamma$ appears at least twice in any case. Combining this with $\beta^{(1)} \cdots \beta^{(d)} = f_n(x,y) = \pm 1$ gives
    \begin{equation}\label{eq:betajupperbound}
       \abs{ \beta^{(j)} } \leq \prod_{ \substack{ i = 1 \\ i\neq j } }^d \frac{ 2 }{ \abs{y} \abs{ \alpha^{(j)} - \alpha^{(i)} } } = \landauO{ \frac{1}{ \abs{y} \gamma_\epsilon(j)^{n} } }.
    \end{equation}
    
    Furthermore, if we add and then substract $G_j(n)\,y$ from $\beta^{(i)}$, for $i \neq j$, then we have
    \begin{align*}
       \log\abs{\beta^{(i)}} &= \log\abs{ -\alpha^{(i)}y + \beta^{(j)} + \alpha^{(j)}y } \\
       &= \log\abs{ -\eta_j^{(i)}y + \beta^{(j)} + \eta_j^{(j)} }.
    \end{align*}
    
    Together with Equations~\eqref{eq:betajupperbound} and \eqref{eq:etabounds}, this gives a representation for $\log\abs{\beta^{(i)}}$,
    \begin{equation}\label{eq:betaibylogy}
       \log\abs{ \beta^{(i)} } = \log\abs{y} + \log\abs{ \eta_j^{(i)} } + \landauO{ \frac{1}{ \abs{y} \gamma_\epsilon(j)^n } } \qquad i \in \set{1, \ldots d}\backslash\set{j}.
    \end{equation}
    
    A second representation can be obtained via the group $G = \langle -1, \eta_1^{(i)}, \ldots, \eta_{d-1}^{(i)} \rangle$. Since $\beta^{(i)} \in \ZZ[\alpha^{(1)}]^\times$, there are integers $b_1, \ldots, b_{d-1}$ such that for the index $I = [\ZZ[\alpha^{(1)}]^\times : G]$ the relation
    \begin{equation}\label{eq:betaibyetas}
        \log\abs{\beta^{(i)}} = \frac{b_1}{I} \log\abs{\eta_1^{(i)}} + \cdots + \frac{b_{d-1}}{I} \log\abs{\eta_{d-1}^{(i)}} \qquad i \in \set{1, \ldots d}\backslash\set{j}
    \end{equation}
    holds. By comparing both representations, we want to derive a lower bound for $\log\abs{y}$. 

\subsection{A double-exponential lower bound for the solution.}

    We solve Equation~\eqref{eq:betaibyetas} using Cramer's rule and get
    \begin{equation}\label{eq:cramer}
        R \frac{b_k}{I} = u_k \log\abs{y} + v_k + \landauO{ \frac{ n^{d-2} }{ \abs{y} \gamma_\epsilon(j)^n } }
    \end{equation}
    for $1 \leq k \leq d-1$, where
    \begin{align*}
        u_k &= \det \left( \log\abs{\eta_1^{(i)}}, \ldots, \log\abs{\eta_{k-1}^{(i)}}, 1, \log\abs{\eta_{k+1}^{(i)}}, \ldots, \log\abs{\eta_{d-1}^{(i)}} \right)_{i \neq j}, \\
        v_k &= \det \left( \log\abs{\eta_1^{(i)}}, \ldots, \log\abs{\eta_{k-1}^{(i)}}, \log\abs{\eta_j^{(i)}}, \log\abs{\eta_{k+1}^{(i)}}, \ldots, \log\abs{\eta_{d-1}^{(i)}} \right)_{i \neq j}.
    \end{align*}

    If we consider for some $\lambda_0, \lambda_1, \ldots, \lambda_{d-1}$ the linear combinations
    \[
        \myvec{b} = \lambda_0 I + \sum_{k=1}^{d-1} \lambda_k b_k, \quad \myvec{u} = \sum_{k=1}^{d-1} \lambda_k u_k, \quad \myvec{v} = \lambda_0 R + \sum_{k=1}^{d-1} \lambda_k v_k,
    \]
    then this preserves Identity~\eqref{eq:cramer}, in the sense that
    \begin{equation}\label{eq:cramerlincomb}
        R \frac{\myvec{b}}{I} = \myvec{u} \log\abs{y} + \myvec{v} + \landauO{ \frac{ n^{d-2} }{ \abs{y} \gamma_\epsilon(j)^n } }.
    \end{equation}

    We now distinguish between different cases for the type $j$ and show that, for a suitable $\myvec{u}$, $\log\abs{y}$ grows exponentially in $n$.

    \emph{Case $j \leq d-2$:}

    If $j \leq d-2$, then the column $\br{ \log\abs{ \eta_j^{(i)} } }_{i\neq j}$ appears twice in $v_k$, so $v_k = 0$, for all $k\neq j$. If we choose $\lambda_0 = \cdots = \lambda_{d-2} = 0$ and $\lambda_{d-1} = 1$, then Equation~\eqref{eq:cramerlincomb} is
    \begin{equation}\label{eq:cramerfornicej}
        R\frac{b_{d-1}}{I} = u_{d-1} \log\abs{y} + v_{d-1} + \landauO{ \frac{ n^{d-2} }{ \abs{y} \gamma_\epsilon(j)^n } },
    \end{equation}
    and since $d-1 \neq j$, we have $v_{d-1} = 0$.
    
    In $u_{d-1}$ we subtract the penultimate row $i=d-1$ from the last row $i=d$. Writing $l_{i'}^{(i)}$ for $\log\abs{ \eta_{i'}^{(i)} }$ and $\myvec{l}_{i'}$ for the corresponding column vector (excluding by context the last two rows), we have
    \[
        u_{d-1} =
        \begin{vmatrix}
            \myvec{l}_1 & \cdots & \myvec{l}_{d-2} & \myvec{1} \\
            l_1^{(d-1)} & \cdots & l_{d-2}^{(d-1)} & 1 \\
            l_1^{(d)}-l_1^{(d-1)} & \cdots & l_{d-2}^{(d)}-l_{d-2}^{(d-1)} & 0
        \end{vmatrix}.
    \]

    The entries in the last row are very small  for all $i=2, \ldots, d-2$:  We use Lemma~\ref{lem:rootapprox}, factor the dominant term $g\gamma^n$ and write $\log\abs{1+x} = x + \landauO{x^2}$, since the remaining term surely have absolute value $<1$ for sufficiently large $n$.
    \begin{align*}
        l_i^{(d)} - l_i^{(d-1)} &= \log\abs{ \frac{ \alpha^{(d)} - G_i(n) }{ \alpha^{(d-1)} - G_i(n) } } \\
        &= \log\abs{ \frac{ g \gamma^n \br{ 1 + \frac{h_d}{g} \br{\frac{\delta_d}{\gamma}}^n - \frac{G_i(n)}{g\gamma^n} + \landausmallo{ \abs{\frac{\delta_{d}}{\gamma}}^n } } }{ g \gamma^n \br{ 1 + \frac{h_{d-1}}{g} \br{\frac{\delta_{d-1}}{\gamma}}^n - \frac{G_i(n)}{g\gamma^n} + \landausmallo{ \abs{\frac{\delta_{d-1}}{\gamma}}^n } } } } \\
        &= \frac{h_d}{g} \br{\frac{\delta_d}{\gamma}}^n - \frac{G_i(n)}{g \gamma^n} - \frac{h_{d-1}}{g} \br{\frac{\delta_{d-1}}{\gamma}}^n + \frac{G_i(n)}{g \gamma^n} \\
        &\quad+ \landausmallo{ \abs{\frac{\delta_d}{\gamma}}^n } + \landauO{ \frac{\max\set{\abs{\delta_d}, \gamma_i}^{2n}}{\gamma^{2n}} },
    \end{align*}
    (the same holds true for $i=1$ if we put $\gamma_1 = 0$). By Condition~\eqref{cond:seconddominantroots} in Theorem~\ref{thm:main} we have that $\gamma_{d-2}^{2n} = \landausmallo{\gamma^n \delta_d^n}$, i.e the $O$-term is absorbed by the $o$-term. Succinctly put, we can say that
    \begin{equation}\label{eq:ld_ld-1approx}
        l_i^{(d)} - l_i^{(d-1)} = \frac{h_d}{g} \br{\frac{\delta_d}{\gamma}}^n + \landausmallo{\abs{\frac{\delta_d}{\gamma}}^n}.
    \end{equation}
 
    We expand $u_{d-1}$ on the last row and separate the explicit term, which we can then factor, from the $\landausmallo{\abs{\frac{\delta_d}{\gamma}}^n}$ in $l_i^{(d)} - l_i^{(d-1)}$ and shift the latter into the error term. The minors, other than the last one which has coefficient $0$ in the expansion, are all of order $\landauO{n^{d-3}}$ by Equation~\eqref{eq:etabounds} and thus
    \[
        u_{d-1} = 
        \frac{h_d}{g} \br{\frac{\delta_d}{\gamma}}^n
        \begin{vmatrix}
            \myvec{l}_1 & \cdots & \myvec{l}_{d-2} & \myvec{1} \\
            l_1^{(d-1)} & \cdots & l_{d-2}^{(d-1)} & 1 \\
            1 & \cdots & 1 & 0
        \end{vmatrix}
        + \landausmallo{ n^{d-3} \abs{\frac{\delta_d}{\gamma}}^n}
    \]

    We then subtract $l_1^{(d-1)}$ times the last row from the penultimate row. With the same arguments as above, except that now the $\delta_{d-1}$ terms cancel, we have for $i = 2, \ldots, d-2$ that
    \begin{align}\label{eq:li-l1approx}
        l_i^{(d-1)} - l_1^{(d-1)} &= \log\abs{\frac{ \alpha^{(d-1)} - G_i(n) }{\alpha^{(d-1)}}} \nonumber \\
        &= -\frac{g_i}{g} \br{\frac{\gamma_i}{\gamma}}^n + \landausmallo{ \frac{\max\set{\abs{\delta_{d-1}}, \gamma_i}^n}{\gamma^n} } 
    \end{align}
    (for $i=1$ the entry is $0$). We expand by the penultimate row. All minors where the factor is $l_i^{(d-1)}-l_1^{(d-1)}$ we can shift into an error term $\landauO{n^{d-3} \br{\frac{\gamma_{d-2}}{\gamma}}^n }$. Explicitly writing only the last minor (with factor $1$) gives
    \[
        u_{d-1}
        = \frac{h_d}{g} \br{\frac{\delta_d}{\gamma}}^n
        \left(-
        \begin{vmatrix}
            \myvec{l}_1 & \cdots & \myvec{l}_{d-2} \\
            1 & \cdots & 1
        \end{vmatrix}
        + \landauO{ n^{d-3} \br{\frac{\gamma_{d-2}}{\gamma}}^n }
        \right) + \landausmallo{n^{d-3} \abs{\frac{\delta_d}{\gamma}}^n }.
    \]

    We then multiply the last row by the constant $l_1^{(d)} + l_1^{(d-1)} = \log\abs{\alpha^{(d)}} + \log\abs{\alpha^{(d-1)}}$, which is of order $\landauTheta{n}$. Going from $l_1^{(d)}$ to $l_i^{(d)}$ for each $i = 2, \ldots, d-2$ is then possible with an error of
    \[
        \log\abs{ \frac{\alpha^{(d)}}{\alpha^{(d)} - G_i(n)} } = -\log\abs{1 - \frac{G_i(n)}{\alpha^{(d)}}} = \landauO{\abs{\frac{G_i(n)}{\alpha^{(d)}}}} = \landauO{ \br{\frac{\gamma_i}{\gamma}}^n }.
    \]

    Similarly, we can go from $l_1^{(d-1)}$ to $l_i^{(d-1)}$. Taken together, this means that
    \[
        \begin{vmatrix}
            \myvec{l}_1 & \cdots & \myvec{l}_{d-2} \\
            1 & \cdots & 1
        \end{vmatrix}
        =
        \landauTheta{\frac{1}{n}}
        \begin{vmatrix}
            \myvec{l}_1 & \cdots & \myvec{l}_{d-2} \\
            l_1^{(d)}+l_1^{(d-1)} & \cdots & l_{d-2}^{(d)}+l_{d-2}^{(d-1)}
        \end{vmatrix}
        + \landauO{n^{d-3} \br{\frac{\gamma_i}{\gamma}}^n }
    \]

    If we now add all the other rows to the last one, the entries sum to $-l_i^{(j)}$ according to Equation~\eqref{eq:etamultiplyto1}. This means that
    \[
        \begin{vmatrix}
            \myvec{l}_1 & \cdots & \myvec{l}_{d-2} \\
            1 & \cdots & 1
        \end{vmatrix}
        =
        -\landauTheta{\frac{1}{n}}
        \begin{vmatrix}
            \myvec{l}_1 & \cdots & \myvec{l}_{d-2} \\
            l_1^{(j)} & \cdots & l_1^{(j)}
        \end{vmatrix}
        + \landauO{n^{d-3} \br{\frac{\gamma_i}{\gamma}}^n }
    \]

    After a suitable swapping of rows, the determinant is exactly the one from Lemma~\ref{lem:etamatrixdet} for $k=d-2$, and thus of order $\landauTheta{n^{d-2}}$. We can absorb the error term and get that
    
    \[
        \begin{vmatrix}
            \myvec{l}_1 & \cdots & \myvec{l}_{d-2} \\
            1 & \cdots & 1
        \end{vmatrix}
        = \pm \landauTheta{n^{d-3}}
    \]
    and thus
    \[
        \abs{u_{d-1}} = \landauTheta{ n^{d-3} \abs{\frac{\delta_d}{\gamma}}^n } + \landausmallo{ n^{d-3} \abs{\frac{\delta_d}{\gamma}}^n } = \landauTheta{n^{d-3} \abs{\frac{\delta_d}{\gamma}}^n}.
    \]

    Going back to Equation~\eqref{eq:cramerfornicej} and plugging in our asymptotic for $u_{d-1}$ gives
    \[
        R \frac{\abs{b_{d-1}}}{I} = \landauTheta{ n^{d-3} \abs{ \frac{\delta_d}{\gamma} }^n } \log\abs{y} + \landauO{ \frac{ n^{d-2} }{ \abs{y} \gamma_\epsilon(j)^n } }.
    \]

    Since $\gamma_\epsilon(j)$ contains the factor $\gamma$ at least twice, the error-term (asymptotically) cannot cancel the $\Theta$-term, i.e.
    \[
        R\frac{\abs{b_{d-1}}}{I} = \landauTheta{ n^{d-3} \abs{\frac{\delta_d}{\gamma}}^n } \log\abs{y}
    \]
    and, in particular, $R\frac{\abs{b_d-1}}{I}$ and thus $\abs{b_{d-1}}$ is non-zero. Since $b_{d-1}$ is an integer, we have $\abs{b_{d-1}} \geq 1$. Furthermore, we have $\frac{R}{I} = \landauOmega{n}$ by Lemma~\ref{lem:regNindex}. Going back to the above equation, this gives
    \begin{equation}\label{eq:logyboundgoodj}
        \log\abs{y} = \landauOmega{n^{-(d-4)} \abs{\frac{\gamma}{\delta_d}}^n  }.
    \end{equation}

    \emph{Case $j=d-1$ or $j=d$:} 
    For $j=d-1$ and $k < d-1$ we again have the column $\log\abs{\eta_j^{(i)}}_{i\neq j} = \myvec{l}_j$ twice in $v_k$ and thus $v_k = 0$. This is not the case for $j=d$, instead we take $\myvec{v} = v_{d-2} - v_{d-3}$. After swapping the third and penultimate columns in $v_{d-3}$, we can join the determinants, and adding every other column to the third-last one gives us, by Equation~\eqref{eq:etamultiplyto1},
    \begin{align*}
        \myvec{v} &=
        \begin{vmatrix}
            \myvec{l}_1 & \cdots & \myvec{l}_{d-4} & \myvec{l}_{d-3} & \myvec{l}_d & \myvec{l}_{d-1}
        \end{vmatrix}
        -
        \begin{vmatrix}
            \myvec{l}_1 & \cdots & \myvec{l}_{d-4} & \myvec{l}_{d} & \myvec{l}_{d-2} & \myvec{l}_{d-1}
        \end{vmatrix} \\
        &=
        \begin{vmatrix}
            \cdots & \myvec{l}_{d-3} & \myvec{l}_d & \myvec{l}_{d-1}
        \end{vmatrix}
        +
        \begin{vmatrix}
            \cdots & \myvec{l}_{d-2} & \myvec{l}_{d} & \myvec{l}_{d-1}
        \end{vmatrix} \\
        &=
        \begin{vmatrix}
            \cdots & (\myvec{l}_{d-3}+\myvec{l}_{d-2}) & \myvec{l}_d & \myvec{l}_{d-1}
        \end{vmatrix}
        =
        \begin{vmatrix}
            \cdots & 0 & \myvec{l}_d & \myvec{l}_{d-1}
        \end{vmatrix}
        = 0.
    \end{align*}

    In both cases, taking $\myvec{v} = v_{d-2} - v_{d-3}$ gives us $\myvec{v} = 0$. We calculate the corresponding $\myvec{u}=u_{d-2}-u_{d-3}$, which, with the analouge of the above calculation, gives us
    \begin{align*}
        \myvec{u} &=
        \begin{vmatrix}
            \myvec{l}_1 & \cdots & \myvec{l}_{d-4} & (\myvec{l}_{d-3}+\myvec{l}_{d-2}) & \myvec{1} & \myvec{l}_{d-1}
        \end{vmatrix} \\
        &=
        \begin{vmatrix}
            \cdots & -\myvec{l}_d & \myvec{1} & \myvec{l}_{d-1}
        \end{vmatrix} \\
        &=
        \begin{vmatrix}
            \cdots & (\myvec{l}_{d-1} - \myvec{l}_d) & \myvec{1} & \myvec{l}_{d-1}
        \end{vmatrix}.
    \end{align*}

    The matrix has the rows $i \in \set{1, \ldots, d-2, d-1, d} \backslash \set{j}$, so the penultimate row is $i=d-2$, while the last row is either $i=d$ or $i=d-1$, depending on whether $j=d-1$ or $j=d$.

    Computing the last entry in the third-last column gives us by Lemma~\ref{sec:lemmas} either
    \begin{align*}
        \log\abs{ \frac{ \alpha^{(d-1)} - G_{d-1}(n) }{ \alpha^{(d-1)} - G_d(n) } } = -\landauTheta{n} \qquad \text{ or } \qquad \log\abs{ \frac{ \alpha^{(d)} - G_{d-1}(n) }{ \alpha^{(d)} - G_d(n) } } = \landauTheta{n},
    \end{align*}
    so, up to sign, a $\landauTheta{n}$. The entries in all other rows $i \leq d-2$ are instead
    \begin{align*}
        \log\abs{ \frac{ \alpha^{(i)} - G_{d-1}(n) }{ \alpha^{(i)} - G_d(n) } } = -\frac{h_d}{g} \br{\frac{\delta_d}{\gamma}}^n + \landausmallo{ \abs{\frac{\delta_d}{\gamma}}^n } = \landauO{ \abs{\frac{\delta_d}{\gamma}}^n }
    \end{align*}
    analogous to Equation~\eqref{eq:ld_ld-1approx}. We expand the determinant $\myvec{u}$ by the third-last column and shift all but the last minor into the error term, which gives
    \[
        \pm \myvec{u} = \landauTheta{n}
        \begin{vmatrix}
            \myvec{l}_1 & \cdots & \myvec{l}_{d-4} & \myvec{1} & \myvec{l}_{d-1}
        \end{vmatrix}
        + \landauO{ n^{d-3} \abs{ \frac{\delta_d}{\gamma} }^n }.
    \]

    The matrix now has only the rows $i \in \set{1, \ldots, d-2}$. We add $-l_{d-1}^{(1)}$ times the penultimate column to the last column. For $i=1$ the entry is then $0$ and for $i=2, \ldots d-2$ this then gives us that
    \[
        \log\abs{ \eta_{d-1}^{(i)} } - \log\abs{ \eta_{d-1}^{(1)} } = -\landauTheta{\br{\frac{\gamma_i}{\gamma}}^n},
    \]
    analogous to Equation~\eqref{eq:li-l1approx}. We expand by this column, the last entry $i=d-2$ dominates and we get
    \[
        \begin{vmatrix}
            \myvec{l}_1 & \cdots & \myvec{l}_{d-4} & \myvec{1} & \myvec{l}_{d-1}
        \end{vmatrix}
        = 
        \landauTheta{\br{\frac{\gamma_{d-2}}{\gamma}}^n}
        \begin{vmatrix}
            \myvec{l}_1 & \cdots & \myvec{l}_{d-4} & \myvec{1}
        \end{vmatrix}
        + \landausmallo{ n^{d-4} \br{\frac{\gamma_{d-2}}{\gamma}}^n }.
    \]

    We multiply the last column by $-l_{d-2}^{(d-3)}-l_{d-1}^{(d-3)}-l_d^{(d-3)} = \landauTheta{-n}$. Using the same argument as in Equation~\eqref{eq:ld_ld-1approx}, we can go from $l_{d-2}^{(d-3)}$ [ $l_{d-1}^{d-3}$, resp. $l_d^{(d-3)}$] to $l_{d-2}^{(i)}$ [$l_{d-1}^{(i)}$, resp. $l_d^{(i)}$] and make an error of order $\landauO{\abs{\frac{\gamma_{d-3}}{\gamma_{d-2}}}^n}$ [$\landauO{\abs{\frac{\gamma_{d-3}}{\gamma}}^n}$]. Taken together, this means that
    \begin{align*}
        \begin{vmatrix}          
            \myvec{l}_1 & \cdots & \myvec{l}_{d-4} & \myvec{1}
        \end{vmatrix}
        = &-\landauTheta{\frac{1}{n}}
        \begin{vmatrix}
            \myvec{l}_1 & \cdots & \myvec{l}_{d-4} & (-\myvec{l}_{d-2}-\myvec{l}_{d-1}-\myvec{l}_d)
        \end{vmatrix} \\
        &+ \landauO{n^{d-4}\abs{\frac{\gamma_{d-3}}{\gamma_{d-2}}}^n}
    \end{align*}

    If we subtract all other columns from the last one, then it sums to $\myvec{l}_{d-3}$ according to Equation~\eqref{eq:etamultiplyto1}. Thus the determinant is of order $\landauTheta{n^{d-3}}$ by Lemma~\ref{lem:etamatrixdet}. Put together, this means that
    \[
        \begin{vmatrix}
            \myvec{l}_1 & \cdots & \myvec{l}_{d-4} & \myvec{1}
        \end{vmatrix}
        = -\landauTheta{n^{d-4}} + \landauO{n^{d-4}\abs{\frac{\gamma_{d-3}}{\gamma_{d-2}}}^n} = -\landauTheta{n^{d-4}}.
    \]

    If we go back one equation further and plug this in, we have
    \begin{align*}
        \begin{vmatrix}
            \myvec{l}_1 & \cdots & \myvec{l}_{d-4} & \myvec{1} & \myvec{l}_{d-1}
        \end{vmatrix}
        &= - \landauTheta{ n^{d-4} \br{\frac{\gamma_{d-2}}{\gamma}}^n } + \landausmallo{ n^{d-4} \br{\frac{\gamma_{d-2}}{\gamma}}^n } \\
        &= - \landauTheta{ n^{d-4} \br{\frac{\gamma_{d-2}}{\gamma}}^n } 
    \end{align*}

    So for $\myvec{u}$, in conjunction with $\abs{\delta_d} < \gamma_{d-2}$ from Condition~\eqref{cond:seconddominantroots}, we get that
    \[
        \abs{\myvec{u}} = \landauTheta{n^{d-3} \br{\frac{\gamma_{d-2}}{\gamma}}^n } + \landauO{ n^{d-3} \abs{\frac{\delta_d}{\gamma}}^n } = \landauTheta{n^{d-3} \abs{\frac{\gamma_{d-2}}{\gamma_d}}^n }.
    \]

    We plug this and $\myvec{v} = 0$ into Equation~\eqref{eq:cramerlincomb} and get
    \[
        R \frac{\abs{\myvec{b}}}{I} = \landauTheta{n^{d-3} \abs{\frac{\gamma_{d-2}}{\gamma}}^n} \log\abs{y} + \landauO{\frac{n^{d-2}}{\abs{y} \gamma_\epsilon(j)^n }}.
    \]

    Again, $\gamma_\epsilon(j)$ contains the factor $\gamma$ (at least) twice, so the error term is asymptotically negligible. So the right-hand side, and hence $\myvec{b}$, is non-zero. We derive $\abs{\myvec{b}} \geq 1$, which together with $\frac{R}{I} = \landauOmega{n}$ implies
    \begin{equation}\label{eq:logyboundbadj}
        \log\abs{y} = \landauOmega{ n^{-(d-4)} \br{ \frac{\gamma}{\gamma_{d-2}} }^n }.
    \end{equation}

    If we compare the two bounds from Equations~\eqref{eq:logyboundgoodj} and \eqref{eq:logyboundbadj}, the worse one, Equation~\eqref{eq:logyboundbadj}, holds regardless of the type $j$ of the solution.

\subsection{An exponential-polynomial upper bound for the solution.}

    Returning to our notation $x-\alpha^{(i)}y = \beta^{(i)}$ for $i \in \set{1, \ldots, d}$, we can also take any three $i=i_1, i_2, i_3$ and eliminate both $x$ and $y$ from these equations, giving the relation
    \[
        \br{ \alpha^{(i_3)} - \alpha^{(i_2)} } \beta^{(i_1)} + \br{ \alpha^{(i_1)} - \alpha^{(i_3)} } \beta^{(i_2)} + \br{ \alpha^{(i_2)} - \alpha^{(i_1)} } \beta^{(i_3)},
    \]
    often called Siegel's equation. The standard application of Baker's method (e.g. \cite{gaal91}) is to rewrite Siegel's identity in an $S$-unit equation and to apply lower bounds to the associated linear form in logarithms. The same approach can also be used for parametrised Thue equations (e.g. \cite{tho90}, for a survey, also see \cite{heub05}) to (try to) derive asymptotic bounds on the size of the solutions. Alternatively, one can try to use Bugeaud and Gy\H{o}ry's explicit bounds \cite{bugy96} directly to obtain the same asymptotic bounds with often much worse numerics due to their more general nature.

    Since it is not too much work, we derive an asymptotic polynomial upper bound for $\log\abs{y}$ ourselves, and set $(i_1, i_2, i_3) = (j, k, l)$ in Siegel's identity for the type $j$ and some $k,l$. If $j \in \set{d-1, d}$, it is advantageous to choose $k, l \not \in \set{d-1, d}$, e.g. $k=1, l=2$, whereas if $j \leq d-2$, it is advantageous to choose $k=d, l=d-1$. 
    
    Dividing by the second addend and subtracting the third (flipping the sign of $\alpha^{(k)}-\alpha^{(j)}$) gives the $S$-unit equation
    \begin{equation}\label{eq:Siegel_S-unit}
        \frac{ \alpha^{(l)}-\alpha^{(k)} }{ \alpha^{(j)}-\alpha^{(l)} }\frac{ \beta^{(j)} }{ \beta^{(k)} } + 1 = \frac{ \alpha^{(j)}-\alpha^{(k)} }{ \alpha^{(j)}-\alpha^{(l)} } \frac{ \beta^{(l)} }{ \beta^{(k)} }.
    \end{equation}

    From Equations~\eqref{eq:betajupperbound} and \eqref{eq:betailowerbound} we obtain
    \[
        \frac{ \beta^{(j)} }{ \beta^{(k)} } = \landauO{ \frac{1}{ \abs{y}^2 \gamma_\epsilon(j)^n \abs{\alpha^{(j)} - \alpha^{(k)}} } }.
    \]
    
    Our choice of $k,l$ ensures optimal conditions in Corollary~\eqref{cor:regulatorbound} in the sense that
    \[
        \frac{ \alpha^{(l)} - \alpha^{(k)} }{ \br{ \alpha^{(j)}-\alpha^{(l)} } \br{ \alpha^{(j)}-\alpha^{(k)} } } = \landauO{ \frac{ \max\set{ \abs{\delta_d}, \gamma_2 }^n }{ \gamma^{2n} } },
    \]
    and thus we get in Equation~\eqref{eq:Siegel_S-unit} that
    \[
        \frac{ \alpha^{(l)}-\alpha^{(k)} }{ \alpha^{(j)}-\alpha^{(l)} }\frac{ \beta^{(j)} }{ \beta^{(k)} } + 1 = 1 + \landauO{ \frac{ \max\set{ \abs{\delta_d}, \gamma_2 }^n }{ \abs{y}^2 \gamma_\epsilon(j)^n \gamma^{2n} } },
    \]
    where the factor $\gamma^n$ appears at least four times in the denominator and the factor in the numerator more than cancels out with another factor in $\gamma_\epsilon(j)^n$. Going back to Equation~\eqref{eq:Siegel_S-unit}, we now have for the right-hand side that its logarithm is
    \begin{equation}\label{eq:linearform}
        \log\abs{ \frac{ \beta^{(l)} }{ \beta^{(k)} } } + \log\abs{ \frac{ \alpha^{(j)}-\alpha^{(k)} }{ \alpha^{(j)}-\alpha^{(l)} } } = \landauO{ \frac{ \max\set{ \abs{\delta_d}, \gamma_2 }^n }{ \abs{y}^2 \gamma_\epsilon(j)^n \gamma^{2n} } }.
    \end{equation}

    We want to apply the following lower bound for linear forms in logarithms due to Baker and Wüstholz\cite{bawh91}.
    \begin{thm}\label{thm:bakerwholz}
        Let $\alpha_1, \ldots, \alpha_n$ be three real algebraic numbers greater than $1$ and $b_1, \ldots, b_n$ be integers. Let $K$ be the number field generated by $\alpha_1, \ldots \alpha_n$ over the rationals $\QQ$ and let $d$ be its degree. Put, for the standard logarithmic Weil height $h$,
        \[
            h'(\alpha) = \frac{1}{d} \max\set{ h(\alpha), \abs{\log\alpha}, 1 }.
        \]

        If $\Lambda = b_1 \log\alpha_1 + \cdots + b_n \log\alpha_n \neq 0$, then
        \[
            \log\abs{\Lambda} > - c(n,d) h'(\alpha_1) \cdots h'(\alpha_n) h'(\myvec{b}),
        \]
        where $\myvec{b} = (b_1 : \cdots : b_n)$.
    \end{thm}

    The coefficients of our linear form are both $1$, so we can disregard their contribution to the lower bound. To bound the second term in Equation~\eqref{eq:linearform}, we use the properties
    \[
        h(\alpha \pm \beta) \leq  h(\alpha) + h(\beta) + \log 2, \quad h(\alpha \cdot \beta) \leq h(\alpha) + h(\beta)
    \]
    of the logarithmic Weil height (for an overview of this notion of height and its various properties, see, for example, Chapter~3 of Waldschmidt's book~\cite{waldBook}.). 
    
    Since $\alpha^{(j)}, \alpha^{(k)}, \alpha^{(l)}$ are conjugates and therefore of equal height, this gives
    \[
        h\br{ \frac{ \alpha^{(j)}-\alpha^{(k)} }{ \alpha^{(j)}-\alpha^{(l)} } } = \landauO{ h(\alpha^{(j)}) }.
    \]

    Using the connection to the Mahler measure of the minimal polynomial $f$, combined with Lemma~\ref{lem:rootapprox}, we can easily conclude $h(\alpha^{(j)}) = \landauO{n}$.

    For the height of $\beta^{(i)}$ we use Equation~\eqref{eq:betaibyetas} and get
    \[
        h(\beta^{(k)}) \leq h(b_1) h(\eta_1^{(k)}) + \cdots  + h(b_{d-1}) h(\eta_{d-1}^{(i)}).
    \]

    Plugging in the definition of $\eta_i^{(k)}$, we have $h(\eta_i^{(k)}) \leq h(\alpha^{(k)}) + h( G_i(n) ) + \log 2$, the height of the integer $G_i(n)$ is the logarithm of its absolute value, and so we have $h(\eta_i^{(k)}) = \landauO{n}$ again.

    For the heights $h(b_i) = \max\set{\log\abs{b_i}, 0}$, we look again at the system of equations~\eqref{eq:betaibyetas}, denote it by $\myvec{\beta} = \myvec{ \eta } \cdot \frac{1}{I} \myvec{ b } $ and consider its inverse problem $\myvec{ \eta^{-1}  \beta } = \frac{1}{I} \myvec{b} $. We can do this: The matrix $\myvec{\eta}$ is the matrix from Lemma~\ref{lem:etamatrixdet} for $k=d-1$ and therefore has a non-zero determinant for $j = d$. For $j \neq d$ we add all the other rows $i\neq j$ to the row $i=d$ and obtain the negative of the row $i=j$ by Equation~\eqref{eq:etamultiplyto1}. This changes at most the sign of the determinant and not the invertibility of $\myvec{\eta}$.

    Considering the system of equations $\myvec{ \eta^{-1} \beta } = \frac{1}{I} \myvec{b}$, we then take the (column-wise) maximum norm $\norm{\cdot}_\infty$, which gives
    \[
        \frac{1}{I} \norm{ \myvec{b} }_\infty = \norm{ \myvec{ \eta^{-1} \beta } }_\infty \leq \norm{ \myvec{ \eta^{-1} } }_\infty \cdot \norm{ \myvec{\beta} }_\infty.
    \]

    We have $I = \landauO{n^{d-2}}$ by Lemma~\ref{lem:regNindex}, $\det \myvec{\eta} = \landauTheta{n^{d-1}}$ by Lemma~\ref{lem:etamatrixdet}, $\log\abs{\beta^{(i)}} = \landauO{ \log\abs{y} + \log\abs{\eta_j^{(i)}} } $ by Equation~\eqref{eq:betaibylogy} and $\log\abs{\eta_j^{(i)}} = \landauO{n}$ by definition and Lemma~\ref{lem:rootapprox}. Taken together, this gives us
    \[
        h(b_i) = \max\set{\log\abs{b_i}, 0} = \landauO{ \max\set{ \log\log\abs{y}, \log n, 0 } },
    \]
    and we can assume $h(b_i) = \landauO{ \log\log\abs{y} }$, as otherwise we would already have $\log\abs{y} = \landauO{n}$.

    If we plug everything into Theorem~\ref{thm:bakerwholz} (the asymptotic bounds don't change for the modified height $h'$), then we get
    \[
        \log\abs{ \log\abs{ \frac{ \beta^{(l)} }{ \beta^{(k)} } } + \log\abs{ \frac{ \alpha^{(j)}-\alpha^{(k)} }{ \alpha^{(j)}-\alpha^{(l)} } } } = - \landauOmega{ n^2 \log\log\abs{y} }.
    \]

    If we compare this with the upper bound of Equation~\eqref{eq:linearform}, we get
    \[
        n \log\abs{y} = \landauO{ n^2 \log\log\abs{y} },
    \]
    which implies $\log\abs{y} = \landauO{ n \log n }$: If the implied constant is $c$, i.e.
    \[
        \log\abs{y} \leq c\, n \log\log\abs{y},
    \]
    then the assertion is true if, for example, the relation $\log\abs{y} \leq 2c\, n \log n$ holds. If instad $\log\abs{y} > 2c\, n \log n$, then we have
    \[
        2c\, n \log n < \log\abs{y} \leq c\, n \log\log\abs{y},
    \]
    which gives $\log n < \frac{1}{2} \log\log\abs{y}$ and therefore $n < \sqrt{\log\abs{y}}$. But going back to the original inequality, then we have
    \[
        \log\abs{y} < c \sqrt{\log\abs{y}} \log\log\abs{y},
    \]
    and therefore $\log\abs{y} = \landauO{1}$, which is even stronger than the assertion.

    In summary, we have the asymptotically almost linear upper bound
    \[
        \log\abs{y} = \landauO{n \log n},
    \]
    and by comparing this with the asymptotically exponential lower bound from Equation~\eqref{eq:logyboundbadj} we get $n = \landauO{1}$, which in turn implies $\log\abs{y} = \landauO{1}$. This concludes our proof of Theorem~\ref{thm:main}.
    
%
%


\begin{bibdiv}
\begin{biblist}

    \bib{baker68}{article}{
        title = {Contributions to the theory of diophantine equations I},
        subtitle = {On the representation of integers by binary forms},
        author = {Baker, A.},
        journal = {Philos. Trans. Royal Soc. A},
        volume = {263},
        year = {1968},
        pages = {173--191}
    }

    \bib{bawh91}{article}{
        title = {Logarithmic forms and group varieties},
        author = {Baker, A.},
        author = {Wüstholz, G.},
        journal = {J. Reine Angew. Math.},
        volume = {442},
        year = {1991},
        pages = {19--62}
    }

    \bib{bugy96}{article}{
        title={Bounds for the solutions of Thue-Mahler equations and norm form equations},
        author={Bugeaud, Y.},
        author = {Gy\H{o}ry, K.},
        journal={Acta Arith.},
        volume={74},
        number={3},
        year={1996},
        pages={273--292}
    }

    \bib{ellison72}{article}{
        title = {The diophantine equation $y^2 + k = x^3$},
        author = {Ellison, W.J.},
        author = {Ellison, F.},
        author = {Pesek, J.},
        author = {Stahl, C.E.},
        journal = {J. Number Theory},
        volume = {4},
        number = {2},
        year = {1972},
        pages = {107--117}
    }


    \bib*{CNT91}{book}{
        title = {Computational Number Theory},
        year = {1991},
        publisher = {De Gruyter},
        address = {Berlin/New York}
    }

    \bib{gaal91}{article}{
        title={On the resolution of some diophantine equations},
        author={Gaál, I.},
        book = {CNT91},
        pages={261--280}
    }

    \bib{gersh31}{article}{
        title={Über die Abgrenzung der Eigenwerte einer Matrix},
        author={Gershgorin, S.},
        journal={Zbl},
        volume={6},
        date={1931},
        pages={749--754}
    }

    \bib{hakolepeti99}{article}{
        title = {Thue equations associated with Ankeny-Brauer-Chowla number fields},
        author = {Halter-Koch, F.},
        author = {Lettl, G.},
        author = {Pethö, A.},
        author = {Tichy, R.F.},
        journal = {J. London Math. Soc.},
        volume = {60},
        number = {1},
        year = {1999},
        pages = {1--20}
    }

    \bib{heub99}{article}{
        title={On Families of Parametrized Thue Equations},
        author={Heuberger, C.},
        journal={J. Number Theory},
        volume={76},
        date={1999},
        pages={45--61}
    }

    \bib{heubtichy99}{article}{
        title = {Effective solution of families of Thue equations containing several parameters},
        author = {Heuberger, C.},
        author = {Tichy, R.F.},
        journal = {Acta Arith.},
        year = {1999},
        volume = {91},
        pages = {147--163}
    }

    \bib{heubConj}{article}{
        title = {On a conjecture of E. Thomas concerning parametrized Thue equations},
        author = {Heuberger, C.},
        journal = {Acta Arith.},
        year = {2001},
        volume = {98},
        number = {4},
        pages = {375--394}
    }

    \bib*{RIMS06}{book}{
        title = {Proceedings of the RIMS symposium “Analytic Number Theory and Surrounding Areas”},
        volume = {1511},
        year = {2006}
    }

    \bib{heub05}{article}{
        title={Parametrized Thue Equations--A survey},
        author={Heuberger, C.},
        book={RIMS06},
        pages={82--91}
    }

    \bib{hilgart22}{article}{
        title = {On a family of cubic Thue equations involving Fibonacci and Lucas numbers},
        author = {Hilgart, T.},
        author = {Vukusic, I.},
        author = {Ziegler, V.},
        journal = {Integers},
        year = {2022},
        volume = {22},
        pages = {Paper No. A31, 20}
    }

    \bib{hilgart23}{article}{
        title = {On families of cubic split Thue equations parametrised by linear recurrence sequences},
        author = {Hilgart, T.},
        journal = {Publ. Math. Debrecen},
        volume = {102},
        year = {2023},
        number = {3-4},
        pages = {439--457}
    }

    \bib{petho87}{article}{
        title = {On the Resolution of Thue Inequalities},
        author = {Pethö, A.},
        journal = {J Symb Comput},
        volume = {4},
        number = {1},
        year = {1987},
        pages = {103--109}
    }

    \bib{pohst77}{article}{
        title={Regulatorabschätzungen für total reelle algebraische Zahlkörper},
        author={Pohst, M.},
        journal={J. Number Theory},
        volume={9},
        date={1977},
        pages={459--492}
    }

    \bib{tho90}{article}{
        title={Complete Solutions to a Family of Cubic Diophantine Equations},
        author={Thomas, E.},
        journal={J. Number Theory},
        volume={34},
        number = {2},
        date={1990},
        pages={235--250}
    }

    \bib{thue09}{article}{
        title = {Über Annäherungswerte algebraischer Zahlen},
        author = {Thue, A.},
        year = {1909},
        journal = {J für die Reine und Angew. Math.},
        pages = {287--305}
    }

    \bib{waldBook}{book}{
        title={Diophantine Approximation on Linear Algebraic Groups},
        author={Waldschmidt, M.},
        year = {2000},
        publisher = {Springer},
        address = {Berlin, Heidelberg}
    }

\end{biblist}
\end{bibdiv}
\end{document}